\documentclass[12pt]{joa2011}
\usepackage{latexsym, amsthm, amscd, graphicx, euscript, float, subfig}
\usepackage{amsmath}
\usepackage{graphicx}

\newcounter{figs}

\newtheorem{theorem}[equation]{Theorem}
\newtheorem{lemma}[equation]{Lemma}
\newtheorem{corollary}[equation]{Corollary}

\theoremstyle{remark}
\newtheorem{definition}[equation]{Definition}
\newtheorem{remark}[equation]{Remark}
\newtheorem{example}[equation]{Example}

\numberwithin{equation}{section}

\theoremstyle{definition}

\numberwithin{equation}{section}

\newcommand{\ci}{\perp\!\!\!\perp}

\newcommand{\ncom}{\newcommand}

\ncom{\beq}{\begin{equation}}
\ncom{\eeq}{\end{equation}}
\ncom{\bea}{\begin{eqnarray*}}
\ncom{\eea}{\end{eqnarray*}}
\ncom{\beqa}{\begin{eqnarray}}
\ncom{\eeqa}{\end{eqnarray}}
\ncom{\nno}{\nonumber}
\ncom{\ds}{\displaystyle}
\ncom{\half}{\frac{1}{2}}
\ncom{\mbx}{\makebox{.25cm}}
\ncom{\hs}{\mbox{\hspace{.25cm}}}
\ncom{\rar}{\rightarrow}
\ncom{\Rar}{\Rightarrow}
\ncom{\noin}{\noindent}
\ncom{\bc}{\begin{center}}
\ncom{\ec}{\end{center}}
\ncom{\sz}{\scriptsize}
\ncom{\rf}{\ref}
\ncom{\s}{\sqrt{2}}
\ncom{\sgm}{\sigma}
\ncom{\Sgm}{\Sigma}
\ncom{\psgm}{\sigma^{\prime}}
\ncom{\dt}{\delta}
\ncom{\Dt}{\Delta}
\ncom{\lmd}{\lambda}
\ncom{\Lmd}{\Lambda}
\ncom{\Th}{\Theta}
\ncom{\e}{\eta}
\ncom{\eps}{\epsilon}
\ncom{\pcc}{\stackrel{P}{>}}
\ncom{\lp}{\stackrel{L_{p}}{>}}
\ncom{\dist}{{\rm\,dist}}
\ncom{\sspan}{{\rm\,span}}
\ncom{\re}{{\rm Re\,}}
\ncom{\im}{{\rm Im\,}}
\ncom{\sgn}{{\rm sgn\,}}
\ncom{\ba}{\begin{array}}
\ncom{\ea}{\end{array}}
\ncom{\hone}{\mbox{\hspace{1em}}}
\ncom{\htwo}{\mbox{\hspace{2em}}}
\ncom{\hthree}{\mbox{\hspace{3em}}}
\ncom{\hfour}{\mbox{\hspace{4em}}}
\ncom{\vone}{\vskip 2ex}
\ncom{\vtwo}{\vskip 4ex}
\ncom{\vonee}{\vskip 1.5ex}
\ncom{\vthree}{\vskip 6ex}
\ncom{\vfour}{\vspace*{8ex}}
\ncom{\norm}{\|\;\;\|}
\ncom{\integ}[4]{\int_{#1}^{#2}\,{#3}\,d{#4}}
\ncom{\vspan}[1]{{{\rm\,span}\{ #1 \}}}
\ncom{\dm}[1]{ {\displaystyle{#1} } }
\ncom{\ri}[1]{{#1} \index{#1}}
\def\theequation{\thesection.\arabic{equation}}
\newtheoremstyle
    {remarkstyle}
    {}
    {11pt}
    {}
    {}
    {\bfseries}
    {:}
    {     }
    {\thmname{#1} \thmnumber{#2} }

\theoremstyle{remarkstyle}

\def\eps{\varepsilon}

\begin{document}

\title{O\lowercase{n the Construction and the Cardinality of Finite $\sigma$-Fields}}

\author{P. Vellaisamy }
\address{P. Vellaisamy, Department of Mathematics,
Indian Institute of Technology Bombay, Powai, Mumbai 400076, INDIA.}
\email{pv@math.iitb.ac.in} 
\author[Sayan Ghosh]{S. Ghosh}
\address{Sayan Ghosh, Department of Mathematics,
Indian Institute of Technology Bombay, Powai, Mumbai 400076, INDIA.}
\email{sayang@math.iitb.ac.in}
\author[M. Sreehari]{M. Sreehari}
\address{M. Sreehari,
6-B, Vrundavan Park, New Sama Road, Vadodara-390024, INDIA.}
\email{msreehari03@yahoo.co.uk}

\thanks{The research of S. Ghosh was supported by UGC, Govt. of India grant F.2-2/98 (SA-I)}
\begin{abstract}
\noindent In this note, we first discuss some properties of generated $\sigma$-fields and a simple approach
to the construction of finite $\sigma$-fields. It is shown that the $\sigma$-field generated by a finite class of  $\sigma$-distinct sets which are also atoms, is the same as the one generated by the  partition induced by them. The range of the cardinality of such a generated  $\sigma$-field  is explicitly obtained. Some typical examples and their complete forms are discussed. We discuss also a simple algorithm to find the exact cardinality of some  particular finite $\sigma$-fields.  Finally, an application of our results to statistics, with regard to independence of events, is pointed out.
\end{abstract}

\maketitle
\vspace*{-0.7cm}

\section{Introduction}
The role of $\sigma$-fields in probability and statistics is well known.  A $\sigma$-field generated by a class of subsets of a give space $\Omega$, also known as the generated $\sigma$-field, is defined to be the intersection of all $\sigma$-fields containing that class. However, a constructive approach to obtain the $\sigma$-field generated by finite number of subsets has not been well studied in the literature. Indeed, no easy constructive method is available. It is a challenging task to construct a $\sigma$-field generated by, say, four arbitrary subsets and to know its exact cardinality. Given a finite class of  $\sigma$-distinct sets (see
Definition 2.2) which are also atoms, we discuss a simple approach to obtain the generated $\sigma$-field based on  the  partition induced by the given class and then look at the $\sigma$-field generated by this partition. In this direction, we first obtain the cardinality of the  induced partition by a finite class of such sets. Using this result, we obtain the cardinality of the generated $\sigma$-field. When all members of the  induced partition are non-empty,  the cardinality of the generated $\sigma$-field is $2^{2^{n}}$, a known result in the literature (see Ash and Doleans-Dade (2000), p~457).  

Let $\mathbb{R}$ denote the set of real numbers. Then the cardinality of $\mathbb{R}$ (also called cardinality of the continuum) is given by $2^{\aleph_{0}},$ where $\aleph_{0}$ is the cardinality of $\mathbb{N},$ the set of natural numbers. One of the fascinating results in measure theory is that there is no $\sigma$-field whose cardinality is countably infinite. In other words, the cardinality of a $\sigma$-field can be either finite or equal to $2^{\aleph_{0}}$ (uncountable)  (Billingsley (1995), p.~34). Our focus is on the cardinality of a $\sigma$-field, generated by a finite class of $n$ sets, and to show that it assumes only particular values within a fixed range. 
To the best of our knowledge, this problem has been addressed only for some special cases (Ash and Doleans-Dade (2000), p.~11).  Several typical examples are discussed to bring out the nature of the finitely generated $\sigma$-fields. An algorithm to find the exact cardinality of some specific $\sigma$-fields of interest is also presented. Finally, we discuss an application to statistics with regard to the independence of some events and establish some new results in this direction.   
\vspace{-0.4cm}
\section{Range for Cardinality of a finite $\sigma$-field}

Let us begin with a simple example. Let $\Omega$ be the given space, and consider two distinct subsets $B$ and $C$ of $\Omega$ such that $B\cup C\neq \Omega$ and $B\cap C \neq \phi.$ We  henceforth suppress the intersection symbol
$\cap$, unless the context demands it. For instance, $A\cap B$ will be denoted by $AB$.  

It is well known that the $\sigma$-field generated by the class $\left\{B,C\right\},$ denoted by $\sigma(B, C)$, is obtained by the usual operations of complementations, unions and intersections, as 
\begin{eqnarray*}
\sigma(B, C) &=& \left\{\phi, \Omega, B, C, B^{c}, C^{c}, B\cup C, B\cup C^{c}, B^{c}\cup C, B^{c}\cup C^{c}, B C, \right. \nonumber \\
& & \left. B C^{c}, B^{c} C, B^{c} C^{c}, (B C^{c})\cup (B^{c} C), (B C) \cup(B^{c} C^{c})\right\}. \nonumber
\end{eqnarray*}
Let $\mathcal{P}_{B,C} = \left\{B C, B^{c} C, B C^{c}, B^{c} C^{c}\right\}$ denote the  partition (see Definition \ref{2.2}) of $\Omega$ induced by $B$ and $C.$ The elements of the $\sigma$-field generated by a partition consist of the empty set $\phi$, the sets in the partition and all possible (finite) unions of them. The resulting class is evidently closed under complementation. Hence, the $\sigma$-field generated by $\mathcal{P}_{B,C}$ is given by 
\begin{eqnarray}\label{eqn2.1}
\sigma(\mathcal{P}_{B,C}) &=& \sigma(B C, B^{c} C, B C^{c}, B^{c} C^{c}) \nonumber \\
&=& \left\{\phi, \Omega, B C, B^{c} C, B C^{c}, B^{c} C^{c}, (B C)\cup (B^{c} C), \right. \nonumber \\ 
& & \left. (B C)\cup (B C^{c}),(B C)\cup (B^{c} C^{c}), (B^{c} C)\cup (B C^{c}), \right. \nonumber \\
& & \left. (B^{c} C)\cup (B^{c} C^{c}), (B C^{c}) \cup (B^{c} C^{c}), \right. \nonumber \\ 
& & \left. (B C)\cup (B^{c} C)\cup (B C^{c}), (B C)\cup (B^{c} C)\cup (B^{c} C^{c}), \right. \nonumber \\
& & \left. (B C)\cup (B C^{c})\cup (B^{c} C^{c}), (B^{c} C)\cup (B C^{c})\cup (B^{c} C^{c})\right\} \\
&=& \left\{\phi, \Omega, B C, B^{c} C, B C^{c}, B^{c} C^{c}, C, B, (B C)\cup (B^{c} C^{c}),\right. \nonumber \\
& & \left.(B^{c} C)\cup (B C^{c}), B^{c}, C^{c}, B\cup C, B^{c}\cup C, C^{c}\cup B, B^{c}\cup C^{c} \right\}, \nonumber 
\end{eqnarray}
as some unions in (\ref{eqn2.1}) admit simple forms which can be obtained using De Morgan's laws.
It is interesting to note  that $\sigma(B,C) = \sigma(\mathcal{P}_{B,C}).$ Note that $|\sigma(B,C)| = 16.$ 

\noindent Consider next the case when $B C = \phi$. In this case, 
\begin{eqnarray*}
\mathcal{P}_{B, C} &=& \left\{ B C^{c}, B^{c} C, B^{c} C^{c}\right\} = \left\{B, C, B^{c} C^{c}\right\}; \nonumber \\
\sigma(\mathcal{P}_{B,C})
&=& \left\{\phi, \Omega, B, C, B^{c} C^{c}, B\cup C, B\cup (B^{c} C^{c}), C\cup (B^{c} C^{c}) \right\} \nonumber \\
&=& \left\{\phi, \Omega, B, C, B^{c} C^{c}, B\cup C, C^{c}, B^{c} \right\}. \nonumber \\
&=&\sigma(B, C)
\end{eqnarray*}
and $|\sigma(B,C)| = 8$ here. Thus, when $B C = \phi$ also, we have $\sigma(B,C) = \sigma(\mathcal{P}_{B,C}).$ \\

 Consider next the cardinality of $\mathcal{P}_{B,C}$, denoted by $|\mathcal{P}_{B,C}|,$ where $B$ and $C$ are arbitrary subsets of $\Omega.$ Two cases arise: (i) $B\cup C = \Omega$ and (ii) $B\cup C\neq \Omega.$ \\
Case (i): If $B C = \phi,$ then $\mathcal{P}_{B,C} = \left\{B,C\right\}$. If $B C\neq \phi,$ then $\mathcal{P}_{B,C} = \left\{B C,B^{c} C,B C^{c}\right\}$ or $\left\{B,B^{c} C\right\}$ (if $B\subset C$) or $\left\{C,B C^{c}\right\}$ (if $C\subset B$). Hence, in this case,  $|\mathcal{P}_{B,C}| \in \{2, 3\}.$ 

\noindent Case (ii): If $B C = \phi,$ then $\mathcal{P}_{B,C} = \left\{B,C,B^{c} C^{c}\right\}$. If $B C\neq \phi,$ then $\mathcal{P}_{B,C} = \left\{B C,B^{c} C,B C^{c},B^{c} C^{c}\right\}$ or $\left\{B,B^{c} C,B^{c} C^{c}\right\}$ (if $B\subset C$) or $\left\{C,B C^{c},B^{c} C^{c}\right\}$ (if $C\subset B$). Hence,  in this case, $|\mathcal{P}_{B,C}| \in \{3, 4\}.$

 \noindent Thus, from cases (i) and (ii), we have  $|\mathcal{P}_{B,C}|\in \left\{2,3,4\right\}$ for arbitrary subsets $B$ and $C$.\\

\noindent It is of interest to know  the cardinality of the  partition based on  $n$ arbitrary sets. To answer the question, we
first introduce the following formal definitions.

\begin{definition} \label{2.1}
We call a class $\mathcal{A} = \left\{A_{1},A_{2},\ldots,A_{n}\right\}$ of sets  $\sigma$-distinct if no set in $\mathcal{A}$ can be obtained from other sets by an operation of union or intersection or complementation. 
\end{definition}

\noindent
Let $A^{1}\equiv A$ and $A^{0}\equiv A^{c}$ henceforth. A formal definition of the partition induced by $n$ sets is the following.

\begin{definition} \label{2.2}
Let $\mathcal{A} = \left\{A_{1},A_{2},\ldots,A_{n}\right\}$ be a $\sigma$-distinct class of subsets of $\Omega.$ Then the finest partition induced by $\mathcal{A},$ denoted by $\mathcal{P}_{A_{1}, A_{2},\ldots,A_{n}},$ is the collection of sets of the form
\begin{equation} \label{eqn2.2}
\mathcal{P}_{A_{1}, A_{2},\ldots,A_{n}} = \left\{\displaystyle\bigcap_{i=1}^{n} A_{i}^{\epsilon_{i}}~\middle|~\epsilon_{i}\in \left\{0,1\right\}, 0\leq\displaystyle\sum_{i=1}^{n}\epsilon_{i}\leq n\right\} = \mathcal{P}_{\mathcal{A}},
\end{equation}
where the empty sets are excluded.
\end{definition}

\begin{remark} \label{rem2.1}
Note that $\mathcal{P}_{A_{1}, A_{2}, A_{1}\cup A_{2}} = \mathcal{P}_{A_{1}, A_{2}, A_{1} A_{2}} = \mathcal{P}_{A_{1}, A_{2}, A_{1}^{c}} = \mathcal{P}_{A_{1}, A_{2}}.$ That is, $\mathcal{P}_{A_{1}, A_{2}, B} = \mathcal{P}_{A_{1}, A_{2}}$ for $B\in \sigma(A_{1},A_{2}).$ Hence, we consider, without loss of generality, only the $\sigma$-distinct sets $A_{1}$ and $A_{2}.$ 
\end{remark}

The following definition is well known.

\begin{definition} Let $\mathcal{A}$ be a collection of sets.
A set $A\in\mathcal{A}$ is called an atom of $\mathcal{A}$ if $B\subseteq A,$ and $B\in\mathcal{A},$ imply  $B = A;$ that is, no proper subset of $A$ belongs to $\mathcal{A}.$ 
\end{definition}

\noindent We start with a simple fact.

\begin{lemma} \label{lem2.1}
If $\{A_{1},A_{2},\ldots,A_{n}\}$ themselves form a partition of $\Omega,$ then $\mathcal{P}_{A_{1}, A_{2},\ldots,A_{n}} = \\ \left\{A_{1},A_{2},\ldots,A_{n}\right\}$ and $|\mathcal{P}_{A_{1}, A_{2},\ldots,A_{n}}| = n.$
\end{lemma}

\begin{proof} Since $\mathcal{A} = \left\{A_{1},A_{2},\ldots,A_{n}\right\}$ forms a partition of $\Omega,$
 $A_1^{c}\ldots A_n^{c} = \phi$ and all intersections of $A_j$'s order $2$ to $n$ are empty. This implies that all the sets in $\mathcal{P}_{A_{1}, A_{2},\ldots,A_{n}}$ with $2\leq\displaystyle\sum_{i=1}^{n}\epsilon_{i}\leq n$ are empty. Consider next the sets of the form $A_1^{\epsilon_1}\dots A_n^{\epsilon_n}$ with $\displaystyle\sum_{i=1}^{n}\epsilon_{i}=1$. Then, for example,
 $A_{1}A_{2}^{c}\ldots A_{n}^{c}= A_1 \bigcap (\displaystyle\bigcup_{i=2}^{n}A_{i})^{c} = A_1 \neq \phi, $
 since $\displaystyle\bigcup_{i=1}^{n}A_{i} = \Omega$.  Thus, $\mathcal{P}_{A_{1}, A_{2},\ldots,A_{n}} = \left\{A_{1},A_{2},\ldots,A_{n}\right\}$ and hence $|\mathcal{P}_{A_{1}, A_{2},\ldots,A_{n}}| = n.$ 
 \end{proof}

 \noindent The following example clearly shows the nature of the partition and its cardinality.

\begin{example}
Consider the class $\left\{A,B,C,D\right\}$ of 4 subsets of $\Omega.$ Then the  partition $\mathcal{P}_{A,B,C,D}$ induced by them is
\begin{eqnarray*}
\mathcal{P}_{A,B,C,D} &=& \left\{ABCD, A^{c}BCD, AB^{c}CD, ABC^{c}D, ABCD^{c}, A^{c}B^{c}CD, A^{c}BC^{c}D, A^{c}BCD^{c}, \right. \nonumber \\
& & \left. AB^{c}C^{c}D, AB^{c}CD^{c}, ABC^{c}D^{c}, A^{c}B^{c}C^{c}D, A^{c}B^{c}CD^{c}, A^{c}BC^{c}D^{c}, AB^{c}C^{c}D^{c}, \right. \nonumber \\ 
& & \left. A^{c}B^{c}C^{c}D^{c} \right\}. \nonumber
\end{eqnarray*}
If $A,B,C$ and $D$ themselves form a partition of $\Omega,$ then $\mathcal{P}_{A,B,C,D} = \left\{A,B,C,D\right\}$  and $|\mathcal{P}_{A,B,C,D}|$ = 4. 
If $A\cup B\cup C\cup D \neq \Omega,$ there may be two cases as follows. 
\begin{enumerate}
\item[(a)] If $ABCD\neq \phi$ and all other elements of $\mathcal{P}_{A,B,C,D}$ are also non-empty, then \\
$|\mathcal{P}_{A,B,C,D}| = 2^{4} = 16.$
\item[(b)] If $ABCD = \phi,$ there can be the following sub cases :
\begin{enumerate}
\item[(i)] If only $ABCD = \phi,$ then $|\mathcal{P}_{A,B,C,D}| = 15$ and 
\begin{eqnarray*}
\mathcal{P}_{A,B,C,D} &=& \left\{A^{c}BCD, AB^{c}CD, ABC^{c}D, ABCD^{c}, A^{c}B^{c}CD, A^{c}BC^{c}D, A^{c}BCD^{c}, \right. \nonumber \\
& & \left. AB^{c}C^{c}D, AB^{c}CD^{c}, ABC^{c}D^{c}, A^{c}B^{c}C^{c}D, A^{c}B^{c}CD^{c}, A^{c}BC^{c}D^{c}, AB^{c}C^{c}D^{c}, \right. \nonumber \\
& & \left. A^{c}B^{c}C^{c}D^{c} \right\}. \nonumber
\end{eqnarray*}
\item[(ii)] If only $BCD = \phi,$ along with the implied case (i) (not mentioned later), then $|\mathcal{P}_{A,B,C,D}| = 14$ and
\begin{eqnarray*}
\mathcal{P}_{A,B,C,D} &=& \left\{AB^{c}CD, ABC^{c}D, ABCD^{c}, A^{c}B^{c}CD, A^{c}BC^{c}D, A^{c}BCD^{c}, AB^{c}C^{c}D, \right. \nonumber \\
& & \left. AB^{c}CD^{c}, ABC^{c}D^{c}, A^{c}B^{c}C^{c}D, A^{c}B^{c}CD^{c}, A^{c}BC^{c}D^{c}, AB^{c}C^{c}D^{c}, A^{c}B^{c}C^{c}D^{c} \right\}. \nonumber
\end{eqnarray*}
\item[(iii)] If only $BCD = ACD = \phi,$ then $|\mathcal{P}_{A,B,C,D}| = 13$ and
\begin{eqnarray*}
\mathcal{P}_{A,B,C,D} &=& \left\{ABC^{c}D, ABCD^{c}, A^{c}B^{c}CD, A^{c}BC^{c}D, A^{c}BCD^{c}, AB^{c}C^{c}D, AB^{c}CD^{c}, \right. \nonumber \\
& & \left. ABC^{c}D^{c}, A^{c}B^{c}C^{c}D, A^{c}B^{c}CD^{c}, A^{c}BC^{c}D^{c}, AB^{c}C^{c}D^{c}, A^{c}B^{c}C^{c}D^{c} \right\}. \nonumber
\end{eqnarray*}
\item[(iv)] If only $BCD = ACD = ABD = \phi,$ then $|\mathcal{P}_{A,B,C,D}| = 12$ and
\begin{eqnarray*}
\mathcal{P}_{A,B,C,D} &=& \left\{ABCD^{c}, A^{c}B^{c}CD, A^{c}BC^{c}D, A^{c}BCD^{c}, AB^{c}C^{c}D, AB^{c}CD^{c}, ABC^{c}D^{c}, \right. \nonumber \\
& & \left. A^{c}B^{c}C^{c}D, A^{c}B^{c}CD^{c}, A^{c}BC^{c}D^{c}, AB^{c}C^{c}D^{c}, A^{c}B^{c}C^{c}D^{c} \right\}. \nonumber
\end{eqnarray*}
\item[(v)] If only $BCD = ACD = ABD = ABC = \phi,$ then $|\mathcal{P}_{A,B,C,D}| = 11$ and
\begin{eqnarray*}
\mathcal{P}_{A,B,C,D} &=& \left\{A^{c}B^{c}CD, A^{c}BC^{c}D, A^{c}BCD^{c}, AB^{c}C^{c}D, AB^{c}CD^{c}, ABC^{c}D^{c}, \right. \nonumber \\
& & \left. A^{c}B^{c}C^{c}D, A^{c}B^{c}CD^{c}, A^{c}BC^{c}D^{c}, AB^{c}C^{c}D^{c}, A^{c}B^{c}C^{c}D^{c} \right\}. \nonumber
\end{eqnarray*}
\item[(vi)] If only $ABD = ABC = CD = \phi,$ then $|\mathcal{P}_{A,B,C,D}| = 10$ and
\begin{eqnarray*}
\mathcal{P}_{A,B,C,D} &=& \left\{A^{c}BC^{c}D, A^{c}BCD^{c}, AB^{c}C^{c}D, AB^{c}CD^{c}, ABC^{c}D^{c}, A^{c}B^{c}C^{c}D, \right. \nonumber \\
& & \left. A^{c}B^{c}CD^{c}, A^{c}BC^{c}D^{c}, AB^{c}C^{c}D^{c}, A^{c}B^{c}C^{c}D^{c} \right\}. \nonumber
\end{eqnarray*}
\item[(vii)] If only $ABC = CD = BD = \phi,$ then $|\mathcal{P}_{A,B,C,D}| = 9$ and
\begin{eqnarray*}
\mathcal{P}_{A,B,C,D} &=& \left\{A^{c}BCD^{c}, AB^{c}C^{c}D, AB^{c}CD^{c}, ABC^{c}D^{c}, A^{c}B^{c}C^{c}D, A^{c}B^{c}CD^{c}, \right. \nonumber \\
& & \left. A^{c}BC^{c}D^{c}, AB^{c}C^{c}D^{c}, A^{c}B^{c}C^{c}D^{c} \right\}. \nonumber
\end{eqnarray*}
\item[(viii)] If only $CD = BD = BC = \phi,$ then $|\mathcal{P}_{A,B,C,D}| = 8$ and
\begin{eqnarray*}
\mathcal{P}_{A,B,C,D} &=& \left\{AB^{c}C^{c}D,  AB^{c}CD^{c}, ABC^{c}D^{c}, A^{c}B^{c}C^{c}D, A^{c}B^{c}CD^{c}, A^{c}BC^{c}D^{c}, \right. \nonumber \\
& & \left. AB^{c}C^{c}D^{c}, A^{c}B^{c}C^{c}D^{c} \right\}. \nonumber
\end{eqnarray*}
\item[(ix)] If only $CD = BD = BC = AD = \phi,$ then $|\mathcal{P}_{A,B,C,D}| = 7$ and
\begin{eqnarray*}
\mathcal{P}_{A,B,C,D} &=& \left\{AB^{c}CD^{c}, ABC^{c}D^{c}, A^{c}B^{c}C^{c}D, A^{c}B^{c}CD^{c}, A^{c}BC^{c}D^{c}, \right. \nonumber \\
& & \left. AB^{c}C^{c}D^{c}, A^{c}B^{c}C^{c}D^{c} \right\}. \nonumber
\end{eqnarray*}
\item[(x)] If only $CD = BD = BC = AD = AC = \phi,$ then $|\mathcal{P}_{A,B,C,D}| = 6$ and
\begin{eqnarray*}
\mathcal{P}_{A,B,C,D} &=& \left\{ABC^{c}D^{c}, A^{c}B^{c}C^{c}D, A^{c}B^{c}CD^{c}, A^{c}BC^{c}D^{c}, AB^{c}C^{c}D^{c}, A^{c}B^{c}C^{c}D^{c} \right\}. \nonumber
\end{eqnarray*}
\item[(xi)] If only $CD = BD = BC = AD = AC = AB = \phi,$ then $|\mathcal{P}_{A,B,C,D}| = 5$ and
\begin{eqnarray*}
\mathcal{P}_{A,B,C,D} &=& \left\{A^{c}B^{c}C^{c}D, A^{c}B^{c}CD^{c}, A^{c}BC^{c}D^{c}, AB^{c}C^{c}D^{c}, A^{c}B^{c}C^{c}D^{c} \right\}. \nonumber
\end{eqnarray*}
\end{enumerate}
\end{enumerate}
Hence, we observe that $|\mathcal{P}_{A,B,C,D}| \in \left\{4, 5, 6, \ldots, 16\right\}.$ 
\end{example}
   
\begin{remark} 
Two or more different partitions $\mathcal{P}_{A_{1}, A_{2},\ldots,A_{n}}$ may have the same cardinality. For example, if $\mathcal{A}= \{A, B, C, D \}$ such that $A\cup B\cup C\cup D \neq \Omega,$ then $|\mathcal{P}_{A,B,C,D}| = 9$ corresponds to the following distinct cases :
\begin{enumerate}
\item[(i)] If only $AB = CD = \phi$, along with the implied empty intersections, then  
\begin{eqnarray*}
\mathcal{P}_{A,B,C,D} &=& \left\{A^{c}BC^{c}D, A^{c}BCD^{c}, AB^{c}C^{c}D, AB^{c}CD^{c}, A^{c}B^{c}C^{c}D, A^{c}B^{c}CD^{c}, \right. \nonumber \\
& & \left. A^{c}BC^{c}D^{c}, AB^{c}C^{c}D^{c}, A^{c}B^{c}C^{c}D^{c} \right\}. \nonumber
\end{eqnarray*}
\item[(ii)] If only $AB = AC = AD = \phi$, along with the implied empty intersections, then
\begin{eqnarray*}
\mathcal{P}_{A,B,C,D} &=& \left\{A^{c}BCD, A^{c}B^{c}CD, A^{c}BC^{c}D,  A^{c}BCD^{c}, A^{c}B^{c}C^{c}D, \right. \nonumber \\
& & \left. A^{c}B^{c}CD^{c}, A^{c}BC^{c}D^{c}, AB^{c}C^{c}D^{c}, A^{c}B^{c}C^{c}D^{c}\right\}. \nonumber 
\end{eqnarray*}
\item[(iii)] If only $AB =  AC =  BCD = \phi$, along with the implied empty intersections, then 
\begin{eqnarray*}
\mathcal{P}_{A,B,C,D} &=& \left\{A^{c}BC^{c}D, A^{c}BCD^{c}, AB^{c}C^{c}D, A^{c}B^{c}CD, A^{c}B^{c}C^{c}D, \right. \nonumber \\
& & \left. A^{c}B^{c}CD^{c}, A^{c}BC^{c}D^{c}, AB^{c}C^{c}D^{c}, A^{c}B^{c}C^{c}D^{c} \right\}. \nonumber
\end{eqnarray*}
\end{enumerate}
In  the above cases,  the partitions are different,  though their cardinality is same.
\end{remark}

\noindent We next look at the cardinality of $\mathcal{P}_{A_{1}, A_{2},\ldots,A_{n}}$ induced by $\{A_{1},A_{2},\ldots,A_{n} \}$. It looks difficult to argue for the arbitrary sets. However, we have the following result for a fairly large and reasonable class of sets.

\begin{theorem}\label{thm2.1}
Let $\mathcal{A} = \left\{A_{1},A_{2},\ldots,A_{n}\right\}, n\geq 3,$ be a class of $\sigma$-distinct sets of $\Omega$, where each $A_{i}$ is an atom of $\mathcal{A},$ and $\mathcal{P}_{A_{1}, A_{2},\ldots,A_{n}}$ be the  partition induced by $A_{1},A_{2},\ldots,A_{n}.$ Then $|\mathcal{P}_{A_{1}, A_{2},\ldots,A_{n}}| \in \left\{n, n+1, \ldots, 2^{n}\right\}.$ 
\end{theorem}
\begin{proof} Let
\begin{equation} \label{eqn2.3}
\mathcal{Q}_{j} = \left\{\displaystyle\bigcap_{i=1}^{n}A_{i}^{\epsilon_{i}}~\middle|~ \epsilon_{i} \in \left\{0,1\right\}, 1\leq i\leq n,  \displaystyle\sum_{i=1}^{n}\epsilon_{i} = j\right\},
\end{equation}
for $0\leq j\leq n.$ Then $\mathcal{P}_{A_{1}, A_{2},\ldots,A_{n}}$ can be represented as 
\begin{equation}\label{eqn2.4}
\mathcal{P}_{A_{1}, A_{2},\ldots,A_{n}} = \displaystyle\bigcup_{j=0}^{n}\mathcal{Q}_{j},
\end{equation}
where only non-empty elements of $\mathcal{Q}_{j}$'s are considered. This is because some $\mathcal{Q}_{j}$'s may contain empty sets. Note that $\mathcal{Q}_{j}$ consists of at most $\binom{n}{j}$ distinct non-empty sets corresponding to a selection of $j$ of the $n$ $\epsilon_{i}$'s as $1$ and the rest as $0.$ So $|\mathcal{Q}_{j}|\leq \binom{n}{j},$ and hence
\begin{equation}\label{eqn2.5}
|\mathcal{P}_{A_{1}, A_{2},\ldots,A_{n}}| = \left|\displaystyle\bigcup_{j=0}^{n}\mathcal{Q}_{j}\right| \leq\displaystyle\sum_{j=0}^{n}|\mathcal{Q}_{j}| \leq \displaystyle\sum_{j=0}^{n}\binom{n}{j} = 2^{n}.
\end{equation}
Note that the elements of $\mathcal{Q}_{j},$ for $1\leq j\leq n-1$ are all disjoint. Indeed, if $D_{1}\in \mathcal{Q}_{i}$ and $D_{2}\in\mathcal{Q}_{j}$, for some $j\neq i$, then $D_{1}D_{2} = \phi.$ Therefore, an element $E\in \mathcal{Q}_{j}$ being empty or non-empty does not affect the nature (emptiness or non-emptiness) of another element $D$ coming from $\mathcal{Q}_{i}$ or $\mathcal{Q}_{j}$ for all $i\neq j,$ in general. \\
Consider next the four exhaustive cases. 

\noindent Case (i): $\displaystyle\bigcup_{i=1}^{n}A_{i} \neq \Omega,$ $\displaystyle\bigcap_{i=1}^{n}A_{i} \neq \phi.$ In this case, our claim is
 \begin{equation}\label{eqn2.6}
|\mathcal{P}_{A_{1}, A_{2}, \ldots, A_{n}}|\in \left\{n+2, n+3, \ldots, 2^{n}\right\}.
\end{equation}

\noindent Here, both $\mathcal{Q}_{0}$ and $\mathcal{Q}_{n}$ contain non-empty sets, so that $|\mathcal{Q}_{0}| = |\mathcal{Q}_{n}| = 1.$ Suppose now $A_{i}$'s are such that all elements of $\mathcal{Q}_{1}$ to $\mathcal{Q}_{n-2}$ are empty. We show that none of the elements of $\mathcal{Q}_{n-1}$ is empty. Suppose an element of $\mathcal{Q}_{n-1},$ say, $A_{1}^{c}A_{2}\ldots A_{n} = \phi.$  Note first that $\displaystyle\bigcup_{0\leq\sum_{j=3}^{n}\epsilon_{j}\leq n-2} A_{3}^{\epsilon_{3}}\ldots A_{n}^{\epsilon_{n}} = \Omega$, since the LHS is the union of sets in $\mathcal{P}_{A_{3}\ldots A_{n}}$. Hence, 
\begin{equation}\label{eqn2.6n}
(A_{3}\ldots A_{n})^{c} = \displaystyle\bigcup_{0\leq\sum_{j=3}^{n}\epsilon_{j}\leq n-3} A_{3}^{\epsilon_{3}}\ldots A_{n}^{\epsilon_{n}}.
\end{equation} Now,
\begin{eqnarray*}
A_{1}^{c}A_{2} &=& A_{1}^{c}A_{2}A_{3}\ldots A_{n} + A_{1}^{c}A_{2}(A_{3}\ldots A_{n})^{c} \nonumber \\
&=& A_{1}^{c}A_{2}\bigcap\left(\displaystyle\bigcup_{0\leq\sum_{j=3}^{n}\epsilon_{j}\leq n-3} A_{3}^{\epsilon_{3}}\ldots A_{n}^{\epsilon_{n}}\right) ~~(\text{using }\eqref{eqn2.6n})\nonumber \\
&=& \displaystyle\bigcup_{0\leq\sum_{j=3}^{n}\epsilon_{j}\leq n-3} A_{1}^{c}A_{2}A_{3}^{\epsilon_{3}}\ldots A_{n}^{\epsilon_{n}}\ \nonumber \\
&=& \phi, \nonumber
\end{eqnarray*}
since each element $A_{1}^{c}A_{2}A_{3}^{\epsilon_{3}}\ldots A_{n}^{\epsilon_{n}},$ with $1\leq\sum_{j=3}^{n}\epsilon_{j}\leq n-3$ belongs to one of $\mathcal{Q}_{1}$ to $\mathcal{Q}_{n-2}.$ So $A_{2}\subset A_{1},$  a contradiction, since each $A_{i}$ is an atom of $\mathcal{A}$. Also, in this case, $|\mathcal{P}_{A_{1}, A_{2}, \ldots, A_{n}}| = 2 + \binom{n}{n-1} = n + 2.$ \\ 

Let now $2\leq k\leq n-1.$ Consider the case where  all elements of $\mathcal{Q}_{j}$ for $1\leq j\leq k-1$ are non-empty, only $r$ elements of $Q_{k}$ are empty, and all elements of $Q_{l}$ for $k+1\leq l\leq n-1$ are empty. In this case, $|\mathcal{P}_{A_{1}, A_{2}, \ldots, A_{n}}| = 2 + \displaystyle\sum_{j=1}^{k-1}\binom{n}{j} + r,$ where $1\leq r\leq \binom{n}{k}.$ If $k = n - 1$ and $r = \binom{n}{n-1} = n,$ then $|\mathcal{P}_{A_{1}, A_{2}, \ldots, A_{n}}| = 2 + \displaystyle\sum_{j=1}^{n-2}\binom{n}{j} + n = 2^{n}.$ That is, when all elements of $\mathcal{Q}_{0}$ to $\mathcal{Q}_{n}$ are non-empty, we get $|\mathcal{P}_{A_{1}, A_{2}, \ldots, A_{n}}| = 2^{n}.$ Thus, $|\mathcal{P}_{A_{1}, A_{2}, \ldots, A_{n}}|\in \left\{n+2, n+3, \ldots, 2^{n}\right\}.$ 

\noindent Case (ii): $\displaystyle\bigcup_{i=1}^{n}A_{i} = \Omega,$ $\displaystyle\bigcap_{i=1}^{n}A_{i}\neq \phi.$ In this case, we claim
\begin{equation}\label{eqn2.7}
|\mathcal{P}_{A_{1}, A_{2}, \ldots, A_{n}}|\in \left\{n+1, n+2, \ldots, 2^{n}-1\right\}.
\end{equation}
Here $\mathcal{Q}_{0} = \left\{\phi\right\}$ and $\mathcal{Q}_{n}\neq \left\{\phi\right\}$ and so $|\mathcal{Q}_{n}| = 1.$ The argument for the minimum value of $|\mathcal{P}_{A_{1}, A_{2}, \ldots, A_{n}}|$ is the same as that in Case (i). Also, $|\mathcal{P}_{A_{1}, A_{2}, \ldots, A_{n}}|$ is exactly $1$ less than that in Case (i), and so $|\mathcal{P}_{A_{1}, A_{2}, \ldots, A_{n}}|\in \left\{n+1, n+2, \ldots, 2^{n}-1\right\}.$ 

\noindent Case (iii): $\displaystyle\bigcup_{i=1}^{n}A_{i} \neq \Omega,$ $\displaystyle\bigcap_{i=1}^{n}A_{i} = \phi.$ In this case also,  we claim
 \begin{equation}\label{eqn2.8}
|\mathcal{P}_{A_{1}, A_{2}, \ldots, A_{n}}|\in \left\{n+1, n+2, \ldots, 2^{n}-1\right\}.  
\end{equation}
Here $\mathcal{Q}_{0}\neq \left\{\phi\right\}$ and $|\mathcal{Q}_{0}| = 1,$ while $\mathcal{Q}_{n} = \left\{\phi\right\}.$  
Suppose $A_{i}$'s are such that all elements of $\mathcal{Q}_{2}$ to $\mathcal{Q}_{n-1}$ are empty. Let now an element of $\mathcal{Q}_{1},$ say, $A_{1}A_{2}^{c}\ldots A_{n}^{c} = \phi.$ Then 
$A_{1} = \displaystyle\bigcup_{0\leq\sum_{i=2}^{n}\epsilon_{i}\leq n-1} \left(A_{1}A_{2}^{\epsilon_{2}}\ldots A_{n}^{\epsilon_{n}}\right)
= \phi,$ by assumptions, which is a contradiction. Thus, all elements of $\mathcal{Q}_{1}$ are non-empty, so that $|\mathcal{P}_{A_{1}, A_{2}, \ldots, A_{n}}| = 1 + \binom{n}{1} = n + 1$ in this case. 
The rest of the arguments are similar to those in Case (i). Hence, $|\mathcal{P}_{A_{1}, A_{2}, \ldots, A_{n}}|\in \left\{n+1, n+2, \ldots, 2^{n}-1\right\}.$

Case(iv): $\displaystyle\bigcup_{i=1}^{n}A_{i} = \Omega,$ $\displaystyle\bigcap_{i=1}^{n}A_{i} = \phi.$ Our claim in this case is
\begin{equation}\label{eqn2.9} 
|\mathcal{P}_{A_{1}, A_{2}, \ldots, A_{n}}|\in \left\{n, n+1, \ldots, 2^{n}-2\right\}. 
\end{equation}
Here, $\mathcal{Q}_{0} = \mathcal{Q}_{n} = \left\{\phi\right\}.$ Observe first that it is not necessary that $A_{1},A_{2},\ldots,A_{n}$ themselves form a partition of $\Omega.$ The argument for the minimum value of $|\mathcal{P}_{A_{1}, A_{2}, \ldots, A_{n}}|$ is the same as that in Case (i). Also, $|\mathcal{P}_{A_{1}, A_{2}, \ldots, A_{n}}|$ is exactly $1$ less than that in Case (iii), and so $|\mathcal{P}_{A_{1}, A_{2}, \ldots, A_{n}}|\in \left\{n, n+1, \ldots, 2^{n}-2\right\}.$

\noindent Thus, from all the above cases, $|\mathcal{P}_{A_{1}, A_{2},\ldots,A_{n}}| \in \left\{n, n+1, \ldots, 2^{n}\right\}.$ 
\end{proof}

\begin{remark}
If $R_{j} = \displaystyle\bigcup_{i}B_{i},$ where $B_{i}\in \mathcal{Q}_{j},$ for $0\leq j\leq n,$ then $\Omega = \displaystyle\bigcup_{j=0}^{n}R_{j}.$ Note that $\mathcal{R} = \left\{R_{0},\ldots, R_{n}\right\}$ forms a partition of $\Omega.$ 
\end{remark}

\noindent Our main interest is on the construction and the cardinality of finite $\sigma$-fields. Our approach is via the partition induced by the generating class $ \mathcal A=\{A_{1}, A_{2},\ldots,A_{n}\} $. Note first that $\sigma(\mathcal{P}_{A_{1}, A_{2},\ldots,A_{n}})$ is obtained by including the empty set and taking all the sets in $\mathcal{P}_{A_{1}, A_{2},\ldots,A_{n}},$ and all possible unions taken two at a time, three at a time, and so on till the union of all the sets in $\mathcal{P}_{A_{1}, A_{2},\ldots,A_{n}}.$ The following result justifies our approach.

\begin{theorem}\label{thm2.2}
Let $\mathcal{A} = \left\{A_{1},A_{2},\ldots,A_{n}\right\}$ be a class of $n$ subsets of $\Omega$ that satisfies the conditions of Theorem \ref{thm2.1}. Then
\begin{align} 
 (i)~&\sigma(\mathcal{A}) = \sigma(\mathcal{P}_{\mathcal{A}}) \label{2.10}\\
(ii)~&|\sigma(\mathcal{A})| =2^{|\mathcal{P}_{\mathcal{A}}|}.\nonumber
\end{align}
\end{theorem}

\begin{proof} Let  $\epsilon_{i}\in \left\{0,1\right\}$ for $1\leq i\leq n.$ 
Since $A_{i}\in \mathcal{A},$ we have $A_{i}^{\epsilon_{i}}\in \sigma(\mathcal{A}),$ for $1\leq i\leq n$ and hence $A_{1}^{\epsilon_{1}}\ldots A_{n}^{\epsilon_{n}}\in \sigma(\mathcal{A}).$ Thus, $\mathcal{P}_{\mathcal{A}}\subset \sigma(\mathcal{A})$ and hence $\sigma(\mathcal{P}_{\mathcal{A}})\subseteq \sigma(\mathcal{A}).$ \\
Conversely, for $1\leq i\leq n,$
\begin{eqnarray*}
A_{i} &=& A_{i}\bigcap\left(\displaystyle\bigcup_{\epsilon_{j}:j\neq i}A_{1}^{\epsilon_{1}}\ldots A_{i-1}^{\epsilon_{i-1}}A_{i+1}^{\epsilon_{i+1}}\ldots A_{n}^{\epsilon_{n}}\right) \nonumber \\
&=& \displaystyle\bigcup_{\epsilon_{j}:j\neq i}\left(A_{1}^{\epsilon_{1}}\ldots A_{i-1}^{\epsilon_{i-1}}A_{i}A_{i+1}^{\epsilon_{i+1}}\ldots A_{n}^{\epsilon_{n}}\right) \nonumber \\  
& \in & \sigma(\mathcal{P}_{\mathcal{A}}), \nonumber
\end{eqnarray*}   
since each $A_{1}^{\epsilon_{1}}\ldots A_{i-1}^{\epsilon_{i-1}}A_{i}A_{i+1}^{\epsilon_{i+1}}\ldots A_{n}^{\epsilon_{n}}\in \mathcal{P}_{\mathcal{A}}.$ Thus, $\left\{A_{1},A_{2},\ldots,A_{n}\right\} = \mathcal{A}\subset\sigma(\mathcal{P}_{\mathcal{A}})\Longrightarrow \sigma(\mathcal{A})\subseteq \sigma(\mathcal{P}_{\mathcal{A}}).$  Thus, $\sigma(\mathcal{A}) = \sigma(\mathcal{P}_{\mathcal{A}})$.

 \noindent  Let now $|\mathcal{P}_{A_{1}, A_{2},\ldots,A_{n}}| = k.$
 Since $\sigma(\mathcal{P}_{A_{1}, A_{2},\ldots,A_{n}})$ is the $\sigma$-field obtained by taking all possible
 unions of the sets in $\mathcal{P}_{A_{1}, A_{2},\ldots,A_{n}}$, we have
  \begin{eqnarray*}
|\sigma(\mathcal{P}_{A_{1}, A_{2},\ldots,A_{n}})| &=& 1 + \binom{k}{1} + \binom{k}{2} + \ldots + \binom{k}{k} \nonumber \\
&=& 2^{k}, \nonumber
\end{eqnarray*}
where the unity is added for the empty set.  Hence, $ \sigma(\mathcal{P}_{\mathcal{A}})=2^{|\mathcal{P}_{\mathcal{A}}|}$.    
\end{proof}

\noindent It is known that the cardinality of a finite $\sigma$-field is of the form  $2^{m}$ for some $m \in \mathbb{N}$ (see, for example, Rosenthal (2006), p.~24).
The following corollary, which follows from Theorems
\ref{thm2.1} and \ref{thm2.2}, gives the explicit range for $m$.
\begin{corollary}
Let $\mathcal{F}$ be a finitely generated $\sigma$-field of subsets of $\Omega.$ Then $|\mathcal{F}| = 2^{m},$ for some $n\leq m\leq 2^{n}$ and $n\geq 1$.
\end{corollary}
Suppose $\mathcal{F}=\sigma(\mathcal{B})$ for some class $\mathcal{B}$ of subsets of $\Omega$. Let $\mathcal{A}$ be the largest $\sigma$-distinct subclass of atoms (of $\mathcal{B}$). Then $m=|\mathcal{P}_{A}|$. 
\begin{remark}
(i) Suppose $A_{1}, A_{2},\ldots,A_{n}$ themselves form a partition of $\Omega.$ Then, from  Lemma \ref{lem2.1} and Theorem \ref{thm2.2}, $|\sigma(A_{1}, A_{2},\ldots,A_{n})| = 2^{n}.$ 
 
\noindent (ii) Suppose $\displaystyle\bigcup_{i=1}^{n}A_{i}\neq \Omega$ and $\displaystyle\bigcap_{i=1}^{n}A_{i}\neq \phi.$ Also, if all elements of $\mathcal{Q}_{j}$ for $1\leq j\leq n-1$ are non-empty, then $|\mathcal{P}_{A_{1}, A_{2},\ldots,A_{n}}| = 2^{n}$ (Case (i) of the proof of Theorem \ref{thm2.2}) and hence,
$|\sigma(A_{1}, A_{2},\ldots,A_{n})| = 2^{2^{n}},$ a known result (Ash and Doleans-Dade (2000), p~457.)

\end{remark}

\begin{example} (i) Let $\mathcal{A}= \{A\}$.
 Then $\mathcal{P}_{A} = \left\{A, A^{c}\right\}$ and \\
$$|\mathcal{P}_{A}| = 2,  \sigma(A) = \sigma(\mathcal{P}_{A}) = \left\{\phi, \Omega, A, A^{c}\right\}, |\sigma(A)| = |\sigma(\mathcal{P}_{A})| = 4.$$  

(ii) Let $\mathcal{A}= \{A, B, C\}$ be a collection of three $\sigma$-distinct subsets  such that $A\cup B\cup C \neq \Omega.$ Then $$\mathcal{P}_{A,B,C} = \left\{ABC, A^{c}BC, AB^{c}C, ABC^{c}, A^{c}B^{c}C, A^{c}BC^{c}, AB^{c}C^{c}, A^{c}B^{c}C^{c}\right\}.$$ 
 The different possible values of $|\mathcal{P}_{A,B,C}|$ are listed in the following cases: 
\begin{enumerate}
\item[(i)] If $ABC \neq \phi,$ then $\mathcal{P}_{A,B,C} = \left\{ABC, A^{c}BC, AB^{c}C, ABC^{c}, A^{c}B^{c}C, A^{c}BC^{c}, \right. \\
\left. AB^{c}C^{c}, A^{c}B^{c}C^{c}\right\}.$ 
Also, $|\mathcal{P}_{A,B,C}| = 8$ and $|\sigma(A,B,C)| = |\sigma(\mathcal{P}_{A,B,C})| = 2^{8} = 256.$ 
\item[(ii)] If only $ABC = \phi,$ then $\mathcal{P}_{A,B,C} = \left\{A^{c}BC, AB^{c}C, ABC^{c}, A^{c}B^{c}C, A^{c}BC^{c}, AB^{c}C^{c}, \right. \\ \left. A^{c}B^{c}C^{c}\right\}.$ 
Also, $|\mathcal{P}_{A,B,C}| = 7$ and $|\sigma(A,B,C)| = |\sigma(\mathcal{P}_{A,B,C})| = 2^{7} = 128.$   
\item[(iii)] If  $AB = \phi,$ then $\mathcal{P}_{A,B,C} = \left\{A^{c}BC, AB^{c}C, A^{c}B^{c}C, A^{c}BC^{c}, \right. \\ \left. AB^{c}C^{c}, A^{c}B^{c}C^{c}\right\}.$ 
Also, $|\mathcal{P}_{A,B,C}| = 6$ and $|\sigma(A,B,C)| = |\sigma(\mathcal{P}_{A,B,C})| = 2^{6} = 64.$   
\item[(iv)] If  $AB = AC = \phi,$ then $\mathcal{P}_{A,B,C} = \left\{A^{c}BC, A^{c}B^{c}C, A^{c}BC^{c}, \right. \\
\left. AB^{c}C^{c}, A^{c}B^{c}C^{c}\right\}.$ 
Also, $|\mathcal{P}_{A,B,C}| = 5$ and $|\sigma(A,B,C)| = |\sigma\mathcal{P}_{A,B,C})| = 2^{5} = 32.$   
\item[(v)] If $AB = AC = BC = \phi,$ then $\mathcal{P}_{A,B,C} = \left\{A^{c}B^{c}C, A^{c}BC^{c}, AB^{c}C^{c}, \right. \\
\left. A^{c}B^{c}C^{c}\right\}.$ 
Also, $|\mathcal{P}_{A,B,C}| = 4$ and $|\sigma(A,B,C)| = |\sigma(\mathcal{P}_{A,B,C})| = 2^{4} = 16.$ 

\noindent Also, in this case,
\begin{eqnarray*}
\sigma(A,B,C) &=& \sigma(\mathcal{P}_{A,B,C}) \nonumber \\
&=& \left\{\phi, \Omega, A, B, C, A^{c}, B^{c}, C^{c}, A\cup B, A\cup C, B\cup C, \right. \nonumber \\
& & \left. A\cup B\cup C, A^{c}B^{c}C^{c}, A^{c}B^{c}, A^{c}C^{c}, B^{c}C^{c} \right\}. \nonumber 
\end{eqnarray*}
\end{enumerate}
Thus, we see that $| \sigma(A,B,C)| \in \{16, 32, 64, 128, 256 \}$. This result easily follows also from \eqref{eqn2.6} and \eqref{eqn2.8} and  Theorem\ref{thm2.2}.
\end{example}       

\section{Exact cardinality of some finite $\sigma$-fields} 
Let $\overline{n} = \left\{1,2,\ldots,n\right\}.$ Given a class $\mathcal{A} = \left\{A_{1},A_{2},\ldots,A_{n}\right\}$ of $\sigma$-distinct sets, we call henceforth $A_{i}A_{j},i\neq j,$ $2$-factor intersections, $A_{i}A_{j}A_{k},i\neq j\neq k,$ $3$-factor intersections, and so on. Clearly, there is only one $n$-factor intersection, namely $A_{1}A_{2}\ldots A_{n}.$ We denote $A_{{i_1}}A_{{i_2}}  \ldots A_{i_k} = A_{S},$ where $S = \left\{i_1, i_2,\ldots, i_k \right\}.$ We now discuss a simple algorithm to find the exact cardinality of some special $\sigma$-fileds:\\

\noindent Let $\bigcup_{1}^{n}A_i \neq \Omega$, so that $|Q_0|=1.$ Note that an element $A_1A_{2}^{c}\ldots A_{n}^{c} \in Q_1$ is empty if $A_1 \subset (\bigcup_{j \neq 1}^{n}A_j)$.  Also, let $s_1$ denote the number of 
$A_{i}$'s such that $A_{i}\subset \displaystyle\bigcup_{j=1;j\neq i}^{n}A_{j}$. Then, $n-s_1 = |Q_1|$, the cardinality of $Q_1$.

\noindent  Let now  $3\leq l\leq n.$ Suppose  none of the $j$-factor intersections, for $2\leq j\leq l-1,$ is empty and only $k$ of the $l$-factor intersections, say, $A_{L_{1}},\ldots,A_{L_{k}}$ are empty for some $L_{i}\subseteq \overline{n}$ with $|L_{i}| = l, 1\leq i\leq k.$ Define
\begin{equation*}
S_{L_{i}} = \left\{B|L_{i}\subseteq B\subseteq \overline{n}\right\}.
\end{equation*} 
We call $S_{L_{i}}$ the set of indices  of implied empty intersections of $L_{i},$ as $A_{B} = \phi,~\forall~B\in S_{L_{i}}.$ Then
\begin{equation}\label{eqn1}
|\mathcal{P}_{A_{1},A_{2},\ldots,A_{n}}| = 2^{n} - \left|\displaystyle\bigcup_{i=1}^{k}S_{L_{i}}\right|- s_1,
\end{equation} 
if none of the higher order intersections $A_M$ is empty, where $l<|M| \leq n-1$ and $M \not \in \bigcup_{i=1}^{k}S_{L_{i}}.$

\noindent If there exist $M_1, \ldots, M_r$ such that $l<|M_i| \leq n-1$ and $M_i \not \in \bigcup_{i=1}^{k}S_{L_{i}},$ then 
\begin{equation}\label{eqn2}
|\mathcal{P}_{A_{1},A_{2},\ldots,A_{n}}| = 2^{n} - \left|\displaystyle (\bigcup_{i=1}^{k}S_{L_{i}})
 \bigcup (\bigcup_{j=1}^{r}S_{M_{j}} )   \right|- s_1.
\end{equation}
Indeed, it suffices to consider only those $M_j$'s  such that
 \begin{equation*}
  M_j \not \in (\bigcup_{i=1}^{k}S_{L_i}) \bigcup (\bigcup_{r=1}^{j-1}S_{M_r}),
\end{equation*}
for $ 2 \leq j \leq n-1.$\\

\noindent We next consider some particular cases of interest under the case $s_1=0$. \\
(i) When none of the intersections of $A_j$'s  of order two or more is empty,
\begin{equation*}
|\mathcal{P}_{A_{1},A_{2},\ldots,A_{n}}| = 2^{n}.
\end{equation*}
(ii) When only the $(n - 1)$-factor intersections, say $q_{n-1}$ in number,  are empty, then
\begin{equation*}
|\mathcal{P}_{A_{1},A_{2},\ldots,A_{n}}| = 2^{n} - q_{n-1} -1,
\end{equation*} 
since $A_{1}\ldots A_{n}$ is the only implied empty intersection.  \\
(iii) When all the $2$-factor intersections $A_{i}A_{j} = \phi $, then all the collections $\mathcal{Q}_{2}$ to $\mathcal{Q}_{n}$ are empty and in this case all the elements of $\mathcal{Q}_{1}$ are non-empty. Thus,
\begin{equation*}
|\mathcal{P}_{A_{1},A_{2},\ldots,A_{n}}| = n + 1,
\end{equation*} 
since $\displaystyle\bigcup_{i =1}^{n} A_{i}\neq\Omega$ which implies $Q_0$ is nonempty.

\noindent Two typical examples follow.
\begin{example} Let $\Omega= \{5, 6, \ldots, 20 \}$ and consider the sets $ A_1= \{8, 10, 14, 16, 17 \}$,
$ A_2= \{6, 7, 18 \}$, $ A_3= \{7, 8, 9, 14, 16,  19 \}$ and  $ A_4= \{9, 10, 11, 14, 20 \}$.
Then $\displaystyle\bigcup_{i=1}^{4}A_{i} \neq \Omega,$ and only  $A_{1}A_{2} = A_{2}A_{4} =  \phi$. Hence, $L_{1} = \left\{1,2\right\}, L_{2} = \left\{2,4\right\}$ and
\begin{enumerate}
\item[(i)] $S_{L_{1}} = \left\{\{1,2\},\{1,2,3\},\{1,2,4\},\{1,2,3,4\}\right\},$
\item[(ii)] $S_{L_{2}} = \left\{\{2,4\},\{1,2,4\},\{2,3,4\},\{1,2,3,4\}\right\}.$
\end{enumerate}
Thus, $|S_{L_1} \cup S_{L_2}| = 6.$ Also, here $s_1=0.$ Using (\ref{eqn1}), we get 
$|\mathcal{P}_{A_{1},A_{2},A_{3},A_{4}}| = 2^{4} - 6 = 10$  and by Theorem \ref{thm2.2}, $|\sigma(A_{1},A_{2},A_{3},A_{4})| = 2^{10} = 1024.$ 
\end{example}

\begin{example}
Consider the case where $\Omega = \{5, 6, \ldots, 15\}, A_{1} = \left\{8, 9,10\right\}, A_{2} = \left\{7,11, 12, 15\right\},\\ A_{3} = \left\{7,11, 12, 13 \right\}$ and $A_{4} = \left\{8,10, 13, 15 \right\}.$ Then $\displaystyle\bigcup_{i=1}^{4}A_{i} \neq \Omega$ and  $A_{1}A_{2} = A_{1}A_{3} = \phi$ are the
only 2-factor intersections so that $L_{1} = \left\{1,2\right\}, L_{2} = \left\{1,3\right\}$ and 
\begin{enumerate}
\item[(i)] $S_{L_{1}} = \left\{\{1,2\},\{1,2,3\},\{1,2,4\},\{1,2,3,4\}\right\},$
\item[(ii)] $S_{L_{2}} = \left\{\{1,3\},\{1,2,3\},\{1,3,4\},\{1,2,3,4\}\right\},$
\end{enumerate}

\noindent Also, $A_{2}A_{3}A_{4} = \phi$ so that $M = \left\{2,3,4\right\} \not\in S_{L_1} \cup S_{L_2}$  and 
 $S_{M} = \left\{\{2,3,4\},\{1,2,3,4\}\right\}.$
Therefore, $|S_{L_1} \cup S_{L_2} \cup S_M| = 7.$ Note also that $A_{2}\subset \displaystyle\bigcup_{j\neq 2}^{4}A_{j}, A_{3}\subset \displaystyle\bigcup_{j\neq 3}^{4}A_{j}$ and $A_{4}\subset \displaystyle\bigcup_{j\neq 4}^{4}A_{j},$ so that $s_{1} = 3.$
\noindent Using (\ref{eqn2}),  $|\mathcal{P}_{A_{1},A_{2},A_{3},A_{4}}| = 2^{4} - 7 - 3 = 6$ and, by Theorem \ref{thm2.2}, $|\sigma(A_{1},A_{2},A_{3},A_{4})| = 2^{6} = 64.$
\end{example}   

\noindent Finally, we look at the question of cardinality $\sigma(\mathcal{B},{D}),$ where  $D\not\in \mathcal{B}$ and $\mathcal{B} = \sigma(\mathcal{A})$ is a $\sigma$-field. Here, we need to talk about $\mathcal{P} \{\mathcal{B},{D}\},$ where the members  are not $\sigma$-distinct.

\begin{definition} Let $\mathcal{A}= \{A_1, \ldots, A_n\}$ be an arbitrary class of $n$ subsets of $\Omega.$ Then
the  partition $\mathcal{P}_{\mathcal{A}}^{\ast}$ induced by $\mathcal{A}$ is defined as the subclass 
 of $\mathcal{P}_{\mathcal{A}}$, defined in (\ref{eqn2.1}), where each element of $\mathcal{P}_{\mathcal{A}}^{\ast}$ is an atom. That is, the sets in $\mathcal{P}_{\mathcal{A}}$ which can be obtained as the union of other sets, are removed from $\mathcal{P}_{\mathcal{A}}$ to obtain $\mathcal{P}_{\mathcal{A}}^{\ast}.$         
\end{definition}

\noindent For simplicity and continuity, we denote $\mathcal{P}_{\mathcal{A}}^{\ast}$ by $\mathcal{P}_{\mathcal{A}}$ only.

\begin{theorem}
Let $\mathcal{A} = \left\{A_{1},A_{2},\ldots,A_{n}\right\}$ be a class of $\sigma$-distinct sets, and $\mathcal{B} = \sigma(\mathcal{A})$ be the generated $\sigma$-field. Let $D\subset\Omega$ and $D\not\in\mathcal{B}.$ Then
\begin{equation} \label{neqn3.3}
|\sigma(\mathcal{B},D)| = 2^ {|\mathcal{P}_{\mathcal{A}, D}|}.
\end{equation}   
\end{theorem}

\begin{proof}
Let $\mathcal{P}_{\mathcal{A},D}$ be the finest partition induced by the class $\left\{\mathcal{A},D\right\}.$ Then
\begin{eqnarray*}
\mathcal{P}_{\mathcal{A},D} &=& \left\{A_{1}^{\epsilon_{1}}\ldots A_{n}^{\epsilon_{n}}D^{\epsilon_{n+1}}~\middle|~\epsilon_{j}\in\left\{0,1\right\}, 0\leq\displaystyle\sum_{j=1}^{n+1}\epsilon_{j}\leq n+1\right\} \nonumber \\ 
&=& \left\{A_{1}^{\epsilon_{1}}\ldots A_{n}^{\epsilon_{n}}D~\middle|~\epsilon_{j}\in\left\{0,1\right\}, 0\leq\displaystyle\sum_{j=1}^{n}\epsilon_{j}\leq n\right\} \nonumber \\
& & \displaystyle\bigcup\left\{A_{1}^{\epsilon_{1}}\ldots A_{n}^{\epsilon_{n}}D^{c}~\middle|~\epsilon_{j}\in\left\{0,1\right\}, 0\leq\displaystyle\sum_{j=1}^{n}\epsilon_{j}\leq n\right\}. \nonumber
\end{eqnarray*} 
When none of the elements of $\mathcal{P}_{\mathcal{A},D} $ is empty,
\begin{equation}\label{eqn3.3}
|\mathcal{P}_{\mathcal{A},D}| = 2|\mathcal{P}_{\mathcal{A}}|. 
\end{equation}
Note also that $\mathcal{P}_{\sigma(\mathcal{A})} = \mathcal{P}_{\mathcal{A}},$ since $\sigma(\mathcal{A})$ is the collection of all possible unions of the sets in $\mathcal{P}_{\mathcal{A}}.$ Hence,
\begin{equation}\label{eqn3.4}
\mathcal{P}_{\sigma(\mathcal{A}),D} = \mathcal{P}_{\mathcal{A},D}
\end{equation} 
for $D\not\in\sigma(\mathcal{A}).$
Observe also that the sets in $\{\mathcal{A}, D \}$ are $\sigma$-distinct. Now
\begin{eqnarray*}
\sigma(\mathcal{B},D) = \sigma(\mathcal{P}_{\mathcal{B},D}) 
= \sigma(\mathcal{P}_{\sigma(\mathcal{A}),D}) 
= \sigma(\mathcal{P}_{\mathcal{A},D}).
\end{eqnarray*}
Hence, from Theorem \ref{thm2.2},
\begin{eqnarray*}
|\sigma(\mathcal{B},D)| &=& 2^{|\mathcal{P}_{\mathcal{A},D}|},
\end{eqnarray*}
which proves the result.
\end{proof}

\begin{remark} (i) When none of the elements of $\mathcal{P}_{\mathcal{A},D} $ is empty, we have from \eqref{neqn3.3}  and \eqref{eqn3.3},
\begin{eqnarray} \label{neqn3.5}
|\sigma(\mathcal{B},D)| &=& 2^{|\mathcal{P}_{\mathcal{A},D}|} 
= 2^{2|\mathcal{P}_{\mathcal{A}}|} 
= (2^{|\mathcal{P}_{\mathcal{A}}|})^{2} 
= |\mathcal{B}|^{2}.
\end{eqnarray}

\noindent (ii) It follows from (\ref{eqn3.4}) that 
\begin{eqnarray} \label{neqn3.6}
\sigma(\sigma(\mathcal{A}),D) = \sigma(\mathcal{P}_{\sigma(\mathcal{A}),D}) 
= \sigma(\mathcal{P}_{\mathcal{A},D}) 
= \sigma(\mathcal{A},D),
\end{eqnarray}
an expected result.
\end{remark}


\begin{example}
Consider the simplest case, where $\mathcal{B} = \sigma(A) = \left\{\phi, \Omega, A, A^{c}\right\}.$ Let $D\not\in \mathcal{B}$ be such $AD \neq \phi$ and $A\cup D \neq \Omega.$ Then  $\{\mathcal{{B},{D}}\} = \left\{\phi, \Omega, A, A^{c}, D\right\}$  and $\mathcal{P}_{\mathcal{A},D} = \left\{AD, A^{c}D, AD^{c}, A^{c}D^{c}\right\} = \mathcal{P}_{A,D}.$ Note here  all the  elements of $\mathcal{P}_{A,D}$ are non-empty. Hence, $|\mathcal{P}_{\mathcal{B},D}| = |\mathcal{P}_{A,D}| = 2|\mathcal{P}_{A}| = 4$ and $|\sigma(\mathcal{B},D)| = |\mathcal{B}|^{2} = 16.$  
\end{example}

\begin{example}
Consider the case when $A_{1}$ and $A_{2}$ are two $\sigma$-distinct subsets of $\Omega$ such that $A_{1}\cup A_{2}\neq\Omega$ and $A_{1}\cap A_{2} = \phi.$ Let $\mathcal{A} = \left\{A_{1},A_{2}\right\}$ and $\mathcal{B} = \sigma(\mathcal{A}).$ Here $\mathcal{P}_{\mathcal{A}} = \left\{A_{1}^{c}A_{2},A_{1}A_{2}^{c},A_{1}^{c}A_{2}^{c}\right\}$ and $\mathcal{B} = \left\{\phi, \Omega, A_{1}, A_{2}, A_{1}^{c}, A_{2}^{c}, A_{1}\cup A_{2}, A_{1}^{c}\cap A_{2}^{c} \right\}.$ Let $D\not\in \mathcal{B}$ be a non-empty subset of $\Omega,$ so that $\left\{\mathcal{B},D\right\} = \left\{\phi, \Omega, A_{1}, A_{2}, A_{1}^{c}, A_{2}^{c}, A_{1}\cup A_{2}, A_{1}^{c}\cap A_{2}^{c}, D \right\}$ and $\left\{\mathcal{A},D\right\} = \left\{A_{1},A_{2},D\right\}.$ Then $\mathcal{P}_{\mathcal{B},D} = \left\{A_{1}^{c}A_{2}D, A_{1}A_{2}^{c}D, A_{1}^{c}A_{2}^{c}D, A_{1}^{c}A_{2}D^{c}, A_{1}A_{2}^{c}D^{c}, A_{1}^{c}A_{2}^{c}D^{c}\right\} \\ = \mathcal{P}_{\mathcal{A},D}.$ Assume all elements of $\mathcal{P}_{\mathcal{A},D}$ to be non-empty. Then $|\mathcal{P}_{\mathcal{B},D}| = |\mathcal{P}_{\mathcal{A},D}| = 2|\mathcal{P}_{\mathcal{A}}| = 6.$ Hence, $|\sigma(\mathcal{B},D)| = 2^{|\mathcal{P}_{\mathcal{B},D}|} = 2^{2|\mathcal{P}_{\mathcal{A}}|} = |\sigma(\mathcal{A})|^{2} = |\mathcal{B}|^{2} = 64.$  

\end{example}

\section{An Application to Independence of Events} 

Let us consider three events $A, B$ and $C$ such that $A$ is independent of $B$ (denoted by $A\ci B$) and $A$ is independent of $C$ (denoted by $A\ci C$).  It is known that these assumptions neither imply  $A\ci (B\cup C)$ nor $A\ci (BC)$, unless $B$ and $C$ are disjoint (see Whittaker (1990) p.~ 27). However, we have the following lemma.
\begin{lemma}\label{lem1}
Let $A\ci B$ and $A\ci C$. Then $A\ci (B\cup C)$ iff $A\ci (B C).$ 
\end{lemma}
\begin{proof} Consider
\begin{eqnarray*}
P[A\cap (B\cup C)] &=& P[(A B) \cup (A C)] \nonumber \\
&=& P(A B) + P(A C) - P(A B C) \nonumber \\
&=& P(A)[P(B) + P(C)] - P(A B C) \quad\textrm{(by assumption)} \nonumber \\
&=& P(A)[P(B\cup C) + P(B C)] - P(A B C), \nonumber
\end{eqnarray*}
so that
\begin{equation}\label{eq}
P[A \cap (B\cup C)] - P(A)[P(B\cup C)] = P(A)[P(B C)] - P(A B C). 
\end{equation}
Hence, $A\ci (B\cup C)$ (that is, L.H.S. of (\ref{eq}) = 0) iff $A\ci (B C)$ (that is, R.H.S. of (\ref{eq}) = 0).
\end{proof}

\begin{lemma}\label{lem2}
If $A\ci B,$ $A\ci C$ and $A\ci (B\cup C),$ then $A\ci \mathcal{P}_{B,C}.$  
\end{lemma}
\begin{proof} We need to show that under the assumptions, $A \ci E$,  for every $ E \in  \left\{BC, B^{c}C, BC^{c}, B^{c}C^{c}\right\}= \mathcal{P}_{B,C}$. 
Using Lemma \ref{lem1} and the assumptions, we have $A\ci (B C).$ Since, 
\begin{eqnarray*}
P[A \cap (B^{c} C)] &=& P(A C) - P(A B C) \nonumber \\
&=& P(A)P(C) - P(A)P(B C) \nonumber \\
&=& P(A)[P(C) - P(B C)] \nonumber \\
&=& P(A)P(B^{c} C) \nonumber
\end{eqnarray*}
shows that $A\ci (B^{c}\cap C).$  Similarly,   $A\ci (B C^{c})$, by symmetry.
 Further, $A$ is independent of $B\cup C$ implies $A$ is independent of $(B\cup C)^{c} = B^{c} C^{c}.$ 
Hence, $A\ci \mathcal{P}_{B,C}.$ 
\end{proof}

\noindent The following stronger form of independence of three events is used  especially in graphical models
 (see Whittaker (1990) p.~ 27).
\begin{definition}\label{def4.1}
$A\ci [B,C]$ if A is independent of $E,$  for every $E\in \mathcal{P}_{B,C}$.
\end{definition}

\noindent As another application  Theorem \ref{thm2.2}, we have the following result which states that $A\ci [B,C]$
implies the independence of $A \ci E$, where $E \in  \sigma(B,C)$ and  $ 4 \leq |\sigma(B,C)| \leq 16.$ 

\begin{theorem}
The following statements are equivalent.
\begin{enumerate}
\item[(i)] $A\ci [B,C].$ 
\item[(ii)] $A\ci B, ~ A\ci C$ and $A\ci (B\cup C).$
\item[(iii)] $A\ci \sigma(B,C).$
\end{enumerate}
\end{theorem}
\begin{proof} Assume (i) holds. Since the elements of $\mathcal{P}_{B,C} = \left\{BC, B^{c}C, BC^{c},  B^{c}C^{c}\right\}$ are disjoint, $A$ is independent of all possible unions of elements in $\mathcal{P}_{B, C}.$ Note that $B = BC\cup BC^{c}, C = BC\cup B^{c}C$ and $B\cup C = BC\cup B^{c}C\cup BC^{c}.$ So, $A\ci B, ~ A\ci C$ and $A\ci (B\cup C).$ Hence, $(i) \Longrightarrow (ii).$
Assume now (ii) holds. Then by Lemma \ref{lem2}, $A\ci \mathcal{P}_{B,C}.$ Since $\sigma(\mathcal{P}_{B,C})$ consists 
the collection of all possible unions of sets in $\mathcal{P}_{B,C}$, we have $A\ci \sigma(\mathcal{P}_{B,C})=
\sigma(B,C)$, by  Theorem \ref{thm2.2}. So, $(ii) \Longrightarrow (iii).$ Consider now statements $(i)$ and $(iii).$ Since $\mathcal{P}_{B,C} \subset \sigma(B,C)$, $A\ci \sigma (B,C)$
implies $A \ci \mathcal{P}_{B, C}$ and thus $A\ci [B,C]$, by Definition \ref{def4.1}. Hence, $(iii) \Longrightarrow (i).$ Thus, statements $(i), (ii)$ and $(iii)$ are equivalent.
\end{proof}    

\vspace{0.4cm}
\end{document}